\documentclass[12pt,a4paper,oneside]{amsart}
\usepackage{amsfonts, amsmath, amssymb, amsthm, amscd}
\usepackage{anysize}
\usepackage{mathtools}
\marginsize{2cm}{2cm}{1.5cm}{1.5cm}

\usepackage[T2A]{fontenc}
\usepackage[utf8]{inputenc}
\usepackage[english]{babel}
\usepackage{cmap}
\usepackage{enumerate}

\usepackage{xcolor}
\usepackage{hyperref}
\hypersetup{
	colorlinks   = true, 
	urlcolor     = blue, 
	linkcolor    = blue, 
	citecolor    = red 
}
\newcommand{\Href}[2]{\hyperref[#2]{#1~\ref{#2}}}
\usepackage{anysize}
\marginsize{2cm}{2cm}{1.5cm}{1.5cm}

\usepackage{cancel}

\theoremstyle{theorem}
\newtheorem{thm}{Theorem}
\newtheorem{prp}{Proposition}[section]

\newtheorem{conj}{Conjecture}[section]

\theoremstyle{definition}

\newcommand{\st}{:\;}

\newcommand{\Red}{\R^d}
\def\R{{\mathbb R}}%

\newcommand{\ball}[1]{\mathbf{B}^{#1}}

\newcommand{\vertexset}{\mathrm{vert}\ \! }
\providecommand{\parenth}[1]{\left(#1\right)}%
\providecommand{\braces}[1]{\left\{#1\right\}}%
\newcommand{\iprod}[2]{\left\langle#1,#2\right\rangle}%
\newcommand{\conv}{\mathrm{conv}}%
\def\polar{\circ}
\newcommand{\polarset}[1]{{#1}^{\polar}}%
\newcommand{\id}{\mathrm{Id}}
\providecommand{\abs}[1]{\lvert#1\rvert}%
\DeclareMathOperator{\tr}{\mathrm{trace}}

\title{Quantitative Steinitz Theorem: A polynomial bound}
\author{Grigory Ivanov\address{Grigory Ivanov: 
Institute of Science and Technology Austria (IST Austria), 
Klosterneuburg, 3400, Austria}
\email{grimivanov@gmail.com}
\and
M\'arton Nasz\'odi\address{M\'arton Nasz\'odi:
Alfr\'ed R\'enyi Inst. of Mathematics and
Dept. of Geometry, Lor\'and E\"otv\"os University, Budapest}
\email{marton.naszodi@math.elte.hu}
}

\thanks{
M.N. was supported by the J\'anos Bolyai Scholarship of the Hungarian Academy of Sciences as well as the National Research, Development and Innovation Fund (NRDI) grants K119670, K131529 and K147544, and the \'UNKP-22-5 New National Excellence Program of the Ministry for Innovation and Technology from the source of the NRDI.
}

\subjclass[2020]{52A35 (primary), 52A35, 52A27}
\keywords{Helly-type theorem, centroid, John's ellipsoid, Santal\'o point}

\begin{document}
\begin{abstract}
The classical Steinitz theorem states that if the origin belongs to the interior of the convex hull of a set $S \subset \mathbb{R}^d$, then there are at most $2d$ points of $S$ whose convex hull contains the origin in the interior. 
B\'ar\'any, Katchalski, and Pach proved the following quantitative version of Steinitz's theorem.
Let $Q$ be a convex polytope  in $\mathbb{R}^d$ containing the standard Euclidean unit ball $\mathbf{B}^d$.
Then there exist at most $2d$ vertices of $Q$ whose convex hull
$Q^\prime$ satisfies
\[
r \mathbf{B}^d \subset Q^\prime 
\]  
with $r\geq d^{-2d}$. They conjectured that $r\geq c d^{-1/2}$ holds with a universal constant $c>0$. We prove $r \geq \frac{1}{5d^2}$, the first polynomial lower bound on $r$.
Furthermore, we show that $r$ is not be greater than $\frac{2}{\sqrt{d}}$.
\end{abstract}
\maketitle

\section{Introduction}
The goal of this paper is to establish a quantitative version of the following classical result of  
E. Steinitz \cite{steinitz1913bedingt}.
\begin{prp}[Steinitz theorem]
Let the origin belong to the interior of the convex hull of a set $S \subset \Red.$ 
Then there are at most $2d$ points of $S$ whose convex hull contains the origin in the interior. 
\end{prp}


The first quantitative version of this result was obtained in \cite{barany1982quantitative}, where the following statement was proven.
\begin{prp}[Quantitative Steinitz theorem]\label{prp:qst_from_BKP}
There exists a constant $r = r(d) > 0$ such that for any subset $Q$ of $\Red$ whose convex hull contains the Euclidean unit ball $\ball{d},$ there exists a subset $F$ of $Q$ of size at most $2d$
whose convex hull contains the ball $r \ball{d}.$ 
\end{prp} 
It was also shown that $r(d) > d^{-2d}.$

With the exception of the planar case $d=2$ \cite{kirkpatrick1992quantitative, brass1997quantitative,
barany1994exact}, no significant improvement on $r(d)$ has been obtained (see also \cite{de2017quantitative}).  

Now we state the main result of this paper in which we obtain a polynomial bound on $r(d).$

\begin{thm}[Q.S.T. with polynomial bound]\label{thm:QST_monochromatic}
Let $Q$ be a subset of $\Red$  whose convex hull contains the Euclidean unit ball $\ball{d}.$
Then there exist at most $2d$ points of $Q$ whose convex hull
$Q^\prime$ satisfies
\[
\frac{1}{6d^2} \ball{d} \subset Q^\prime.
\]  
\end{thm}
We conjecture the following.
\begin{conj}\label{conj:QST_monochromatic}
There is a constant $c > 0$ such that  in any  subset $Q$ of $\Red$  whose convex hull contains the Euclidean unit ball $\ball{d},$  there are at most $2d$ points whose convex hull
$Q^\prime$ satisfies
\[
\frac{c}{\sqrt{d}} \ball{d} \subset Q^\prime.
\]  
\end{conj}

We  provide  an upper bound on $r(d).$ 

\begin{thm}\label{thm:qst_upper_bound}
Let $u_1, \dots, u_n$ be unit vectors in $\Red.$
Then their absolute convex hull, that is the convex hull of $\pm u_1, \dots, \pm u_n,$  does not contain the ball
$\parenth{\frac{\sqrt{n}}{d} + \varepsilon} \ball{d}$ for any positive $\varepsilon.$
\end{thm}

It follows that if $u_1, \dots, u_m$ form a sufficiently dense subset of the unit sphere (with a large $m$), then their convex hull is almost the unit ball, while for any $n$ of them with  $n \leq 2d$, we have that their convex hull does not contain the ball $\frac{2}{\sqrt{d}}$, which shows that the order of magnitude of $r(d)$ in \Href{Conjecture}{conj:QST_monochromatic} is sharp if the conjecture holds.

We mention the following conjecture which is closely related to \Href{Theorem}{thm:qst_upper_bound}. It can be found in a different formulation in \cite[p.194]{boroczky2004finite}.
\begin{conj}
\label{conj:2d_cups}
Let $\{u_1, \dots, u_{2d}\}$ be unit vectors  in $\Red$.
Then there is a point in the set  
\[
\bigcap\limits_{i=1}^{2d}\{x \in \Red \st \iprod{u_i}{x} \leq 1\}
\]
 with norm $\sqrt{d}$.
\end{conj}

\section{The main steps in the proof of \Href{Theorem}{thm:QST_monochromatic}}

 Since $r(1) = 1$, we will assume that 
 $d \geq 2$ throughout the paper. 

First, we reduce the problem for the polytopal case. 
By the classical Carath\'eodory theorem \cite[p.200]{caratheodory1911variabilitatsbereich},
 any point of a convex hull of a subset $Q$ of $\Red$ can be represented as a convex combination of at most $d +1$ points of $Q.$ Thus, taking a sufficiently dense subset of the unit sphere, we observe that for any $\epsilon \in (0,1)$ and any  set $Q \subset \Red$  whose convex hull contains 
 $\ball{d},$ there is a finite subset $Q_f$ of $Q$ whose convex hull contains the ball
 $(1- \epsilon) \ball{d}.$ Hence,  \Href{Theorem}{thm:QST_monochromatic} follows from the following polytopal version.
 
\begin{thm}\label{thm:QST_polytope_monochromatic}
Let $Q$ be a convex polytope  in $\Red$ containing the Euclidean unit ball $\ball{d}.$
Then there are at most $2d$ vertices of $Q$ whose convex hull
$Q^\prime$ satisfies
\[
\frac{1}{5d^2} \ball{d} \subset Q^\prime.
\]  
\end{thm}

 \Href{Proposition}{prp:qst_from_BKP} was used in \cite{barany1982quantitative} to prove  certain quantitative versions of the Helly theorem. The connection between the quantitative Steinitz result and the quantitative Helly-type result is via polar duality. Recently, the authors of this paper \cite{ivanov2022quantitative} have proposed a new approach to quantitative Helly-type results via sparse approximation of polytopes. The connection between the sparse approximation of polytopes and the quantitative Helly-type result is via polar duality again. 
We state a refined version of the result on the sparse approximation of polytopes obtained by Almendra--Hern\'andez, Ambrus, and Kendall in \cite[Theorem 1]{almendra2022quantitative}.
\begin{prp}[Almendra--Hern\'andez et. al.]\label{prp:ambrus_sparsification}
Let $\lambda  > 0 ,$ and $L \subset \Red$ be a convex polytope such that 
$L \subset - \lambda L.$ Then there exist at most $2d$ vertices of $L$ whose convex hull
$L^\prime$ satisfies
\[
L \subset -(\lambda +2)d \cdot L^\prime.
\]
\end{prp}

Choosing the origin smartly, one can achieve $\lambda  = d.$ 
For instance, the following statement holds. 

\begin{prp}
\label{prp:d_centers}
Let $K$ be a  convex body in $\Red$.
Then the inclusion $(K-c) \subset 
-d (K-c)$ holds for some point $c$ in the interior of $K$,
for example, if $c$ is the centroid of $K$ or of a maximal volume simplex within $K$.
\end{prp}  

We recall that the \emph{polar} of the set $S \subset \Red$ is defined by
\[
\polarset{S} = \braces{x \in \Red \st \iprod{x}{s} \leq 1 \quad \text{for all} \quad s \in S}.
\]

Our idea of the proof of \Href{Theorem}{thm:QST_polytope_monochromatic} is to use duality twice:
We will start with translating  the assertion of the theorem in terms of the polar polytope $\polarset{Q}$ of $Q$. Then we will choose a point $c$  ``deep'' in $\polarset{Q}$ and consider 
$\polarset{\parenth{\polarset{Q} - c}}.$ Roughly speaking, by changing the center of polarity, we obtain a more well-structured convex polytope. Next, we use \Href{Proposition}{prp:ambrus_sparsification} to obtain a sufficiently reasonable bound on $r(d),$
which is not destroyed on the way back to $\polarset{Q}$ and then to $Q.$  

We use $[n]$ 
to denote the sets $\{1, \dots, n\}.$
The convex hull of a set $S$ is denoted by $\conv \ \! S$.
For a non-zero vector $v \in \Red,$  $H_{v}$  denotes the half-space
\[
H_v = \braces{x \in \Red \st \iprod{x}{v} \leq 1}.
\]
We use $\vertexset P$ to denote the vertex set of a polytope $P.$

For the sake of completeness, we provide a shortened original proof of 
\Href{Proposition}{prp:ambrus_sparsification}.
\begin{proof}[Proof of \Href{Proposition}{prp:ambrus_sparsification}]
   The condition $L \subseteq -\lambda L$ ensures that  the origin belongs to the interior of $L$.     Among all simplices with $d$ vertices from the set of vertices of $L$ and one vertex at the origin, consider a simplex $S = \conv\{0,v_1,\ldots,v_d\}$ with maximal volume.   The simplex $S$ can be represented as
    \begin{equation}\label{eq:conv_S}
        S = \braces{ x \in \Red \st x = \alpha_1 v_1 + \ldots + \alpha_d v_d \quad  \textrm{ for }\  \alpha_i \geq 0 \textrm{ and } \sum_{i=1}^d \alpha_i \leq 1 }.
    \end{equation}
    
     Define $P = \sum\limits_{i \in [d]} [-v_i, v_i].$ It is easy to see that
     $P$ is a paralletope that can be represented as
    \begin{equation}\label{eq:conv_P}
        P = \{ x \in \Red \st x = \beta_1 v_1 + \ldots + \beta_d v_d \quad \textrm{ for } \beta_i \in [-1,1] \}.
    \end{equation}
Since $S$ is chosen maximally, equation \eqref{eq:conv_P} shows that for any vertex $v$ of $L$, $v \in P$. By convexity,
    \begin{equation}
        \label{eq:Q_subset_P_ambrus}
        L \subset P.
    \end{equation}
    Let $S^\prime = -2dS + (v_1 + \ldots + v_d)$. By \eqref{eq:conv_S},
    \begin{equation*}
        S^\prime = \braces{x \in \Red \st x = \gamma_1v_1 + \ldots + \gamma_dv_d \quad  \text{ for } \gamma_i \leq 1 \textrm{ and } \sum_{i= 1}^{d} \gamma_i \geq -d },
    \end{equation*}
    which, together with \eqref{eq:conv_P}, yields
    \begin{equation}
        \label{eq:2}
        P \subseteq S^\prime.
    \end{equation}
    Let $y$ be the intersection of the ray emanating from $0$ in the direction $-(v_1 + \dots + v_d)$     and the boundary of $L$.
    By Carath\'eodory's theorem, we can choose $k \leq d$ vertices $\{v_1',\ldots,v_k'\}$ of $L$ such that
    $y \in \conv\{v_1',\ldots,v_k'\}$.
    Set $L^\prime= \conv \{v_1,\ldots,v_d,v_1',\ldots,v_k'\}$.
Clearly, $\frac{v_1 + \dots + v_d}{d} \in S \subset L.$ 
Thus, $0 \in L^\prime,$ and consequently,
    \begin{equation}
        \label{eq:3}
        S \subseteq L^\prime.
    \end{equation}
 Since $L \subset - \lambda L,$ we also have that 
    \begin{equation*}
       \frac{v_1 + \dots + v_d}{d} \in -\lambda [y, 0] \subset - \lambda  L^\prime.
    \end{equation*}
    Combining it with \eqref{eq:Q_subset_P_ambrus}, \eqref{eq:2}, \eqref{eq:3}, we obtain
    \begin{equation} \label{eq:6}
        L \subset P \subset S^\prime
        = -2d S + (v_1 + \dots + v_d)
        \subset -2d  \, L^\prime - \lambda d \, L^\prime
        = -(\lambda + 2) d \, L^\prime,
    \end{equation}
    Completing the proof of \Href{Proposition}{prp:ambrus_sparsification}.
\end{proof}

\section{Proof of \Href{Theorem}{thm:QST_monochromatic}}
As was explained in the previous section, it suffices to prove 
\Href{Theorem}{thm:QST_polytope_monochromatic}, which we proceed to work with.
  
Set $K = \polarset{Q}.$ Since $Q \supset \ball{d},$ $K \subset \ball{d}.$ 
Also, it is easy to see that $K$ is a  convex polytope of the form
\begin{equation}
\label{eq:set_via_hyperplanes}
K = \bigcap\limits_{v \in \vertexset{Q} } H_v,
\end{equation}
containing the origin in its interior. 
By duality, it suffices to show that there are at most $2d$ half-spaces $H_v$ with $v \in \vertexset{Q},$ whose intersection is contained in the ball $5d^2 \ball{d}.$

Let $c$ be a point in the interior of $K$ such that the inclusion 
\[
K-c \subset -d (K-c)
\]
holds.
The existence of $c$ follows from \Href{Proposition}{prp:d_centers}.
Set $L = \polarset{\parenth{K-c}}.$  Clearly,
\begin{equation*}
L \subset -d L.
\end{equation*}
Now, we use \Href{Proposition}{prp:ambrus_sparsification} with $\lambda =d.$ 
We obtain that there are $w_1, \dots, w_m \in \vertexset {L}$ for some integer 
$m$ satisfying $m \leq 2d$ such that
\[
L \subset -(d+2)d \cdot \conv \braces{w_i \st i \in [m]}.
\]
Since $c \in K  \subset \ball{d},$ one has that 
$K-c \subset 2 \ball{d}.$ Consequently, 
$L \supset \frac{1}{2}\ball{d}.$
So,
\[
\frac{1}{2}\ball{d} \subset L \subset -(d+2)d \cdot \conv \braces{w_i \st i \in [m]}.
\]
Considering the polar sets, 
we get
\[
\polarset{\parenth{\conv \braces{w_i \st i \in [m]}}} \subset 2(d+2)d \ball{d}.
\]
Recall that $c$ is an interior point of the polytope $K.$  By \eqref{eq:set_via_hyperplanes}, one has that for any 
$w \in \vertexset{L},$ $H_w = H_v - c$  for some 
 $v \in \vertexset{Q}.$ It means that 
\[
\polarset{\parenth{\conv \braces{w_i \st i \in [m]}}} = \bigcap\limits_{v_i \in [m] } \parenth{H_{v_i} - c}
\]
for corresponding $v_i \in \vertexset{Q}.$ 
Thus,
\[
\bigcap\limits_{v_i \in [m] } H_{v_i}  = \bigcap\limits_{v_i \in [m] } \parenth{H_{v_i} - c} + c
\subset 2(d+2)d \ball{d} +c \subset \parenth{2(d +2)d +1} \ball{d}.
\]
Since $d \geq 2,$ the desired bound for $Q^\prime = \conv \braces{v_i \st i \in [m]}$ follows. 
The proof of \Href{Theorem}{thm:QST_polytope_monochromatic} is complete, which implies 
\Href{Theorem}{thm:QST_monochromatic} as was discussed earlier.

\section{Proof of \Href{Theorem}{thm:qst_upper_bound}}

In this section, we prove Theorem~\ref{thm:qst_upper_bound}, which is a 
dual version of  \cite[Theorem 1.4]{ivanov2022quantitative} and immediately follows from it. For the sake of completeness, we prove Theorem~\ref{thm:qst_upper_bound} here. 
We first state the main ingredient of the proof obtained by K. Ball and M. Prodromou.
\begin{prp}[\cite{Ball2009}, Theorem 1.4]
\label{prp:trace_bound_tight_frame}
Let  vectors  $\{v_1,  \dots, v_n\}\subset \Red$ satisfy $\sum\limits_1^{n} v_i \otimes v_i = \id$.
Then for any positive semi-definite operator $T \colon \Red \to \Red,$
there is a point $p$ in the intersection of the strips 
$\{x \in \Red \st \abs{\iprod{x}{v_i}} \leq 1\}$ satisfying 
$\iprod{p}{T p} \geq \tr  {T}$.
\end{prp}

\begin{proof}[Proof of Theorem~\ref{thm:qst_upper_bound}]
There is nothing to prove if the absolute convex hull 
$\conv \braces{ \pm u_i \st i \in [n]}$ does not contain the origin in its interior. 
So, assume that   $\conv \braces{ \pm u_i \st i \in [n]}$ contains the origin in its interior.
Set  $K = \polarset{\parenth{\conv \braces{\pm u_i \st i \in [n] }}}.$ By duality,
it suffices to show that $K$ contains a point of Euclidean norm 
$\frac{d}{\sqrt{n}}.$ 

Clearly,  $\{u_i \st i \in [n] \}$ spans $\Red.$ 
Consider  $A = \sum\limits_{i \in [n]} u_i \otimes u_i$.
Since the vectors span the space, $A$ is positive definite.  
Using  \Href{Proposition}{prp:trace_bound_tight_frame} with  
$ v_i = A^{-1/2} u_i,  i \in [n],$  
 and  $T  = A^{-1},$ we find a point $p$ in 
 \[
\bigcap\limits_{i \in [n]} \{x \st \abs{\iprod{v_i}{x}} \leq 1\}.
\] such that
\[
\iprod{p}{A^{-1} p} \geq \tr{A^{-1}}.
\]

Denote $q = A^{-1/2} p$.
Then, by the choice of $p,$ 
\[
1 \geq \abs{\iprod{p}{A^{-1/2} u_i}} = 
\abs{\iprod{A^{-1/2}  p}{u_i}} = 
\abs{\iprod{q}{u_i}}. 
\] 
That is, $q \in K$.
On the other hand, 
\[
\abs{q}^2 = \iprod{A^{-1/2} p}{A^{-1/2} p} = 
\iprod{p}{A^{-1} p } \geq \tr{A^{-1}}.
\]
Finally, since  $\tr A = n $ and by the Cauchy--Schwarz inequality,  
one sees that  
  $\tr A^{-1}$ is at least $\frac{d^2}{n}$.
Thus, $\abs{q}  \geq \frac{d}{\sqrt{n}}$.
This completes the proof of Theorem~\ref{thm:qst_upper_bound}.
\end{proof}

\bibliographystyle{amsalpha}
\bibliography{../../work_current/uvolit}
\end{document}